\documentclass[12pt]{amsart}

\usepackage{amsmath, amssymb}
\usepackage[top=25truemm, bottom=25truemm, left=23truemm, right=23truemm]{geometry}
\usepackage{bm}
\usepackage{bbm}
\usepackage{color}
\usepackage{enumerate}
\usepackage{mathtools}
\usepackage{appendix}

\theoremstyle{definition}
\newtheorem{definition}{Definition}[section]

\theoremstyle{plain}
\newtheorem{theorem}[definition]{Theorem}
\newtheorem{lemma}[definition]{Lemma}
\newtheorem{corollary}[definition]{Corollary}
\newtheorem{proposition}[definition]{Proposition}

\newcommand{\AND}{\quad \text{\rm and} \quad}
\newcommand{\EXP}[2]{\mathbb{E}^{#1}\left[ #2 \right]}

\DeclareMathOperator{\supp}{supp}

\renewcommand{\phi}{\varphi}
\renewcommand{\Re}{\mathrm{Re}\hspace{1pt}}

\renewcommand{\epsilon}{\varepsilon}
\renewcommand{\mid}{~;~}

\begin{document}

\subjclass{Primary 11M32, Secondary 11M06}

\keywords{Euler--Zagier multiple zeta-functions, Value-distribution, Hybrid universality}

\title[Value-distribution of Euler--Zagier type multiple zeta-functions]{Value-distribution of Euler--Zagier type multiple zeta-functions with identical arguments on the critical line}
\author[K.~Endo]{K.~Endo}

\begin{abstract}
We consider the value-distribution of the Euler--Zagier type multiple zeta-functions with identical arguments on the critical line. 
Mishou showed that the set of the values of the Euler--Zagier double sum with identical arguments on the critical line is dense in the complex plane under the Riemann Hypothesis. 
In this paper, we generalize Mishou's result to the $r$-fold Euler--Zagier multiple zeta-functions with identical arguments for $r \geq 4$. 

\end{abstract}

\maketitle


\section{Introduction}

The Euler--Zagier multiple zeta-function is defined by
\[
\zeta_r (s_1, \dots, s_r)
= \sum_{1 \leq n_1 < \cdots < n_r} \frac{1}{ {n_1}^{s_1} \cdots {n_r}^{s_r}}.
\]
This series converges absolutely in the region $\Re(s_r) > 1, \Re (s_r + s_{r-1}) > 2, \ldots, \Re(s_1 + \cdots + s_r)>r$, which was proved by Matsumoto \cite{M2002}. 
Akiyama, Egami and Tanigawa \cite{AET2001} and Zhao \cite{Z2000} independently proved that $\zeta_r(s_1, \ldots, s_r)$ has a meromorphic continuation to the whole space $\mathbb{C}^r$. 
As a first step toward understanding the analytic nature of this function, 
it is natural to investigate the case in which all complex variables $s_1, \ldots, s_r$ are specialized to a single variable $s$. 
We write $\zeta_r(s)$ for
\[
\zeta_r(s) = \zeta_r ( \underbrace{s, \ldots, s}_{r} ).
\]
For example, Matsumoto and Sh{\=o}ji \cite{MS2014} carried out numerical computations on the zeros of $\zeta_2(s)$. 
They later studied the two-variable function $\zeta_2(s_1, s_2)$ both numerically and
theoretically in \cite{MS2020}. 
More recently, Matsumoto, Onozuka, and Wakabayashi \cite{MOW2026} obtained more detailed results on the zero distribution of $\zeta_2(s)$ from a theoretical perspective.
Nakamura and Pa\' nkowski \cite{NP20161} studied complex zeros of $\zeta_r (s)$ for $r \geq 2$. 
In the double case, they \cite{NP20162} also proved that $\zeta_2(s)$ has zeros for $\sigma>1$. 
The behavior of $\zeta_r (s)$ on the real axis has been investigated in \cite{K2006, MMT2022}.
In the present paper, we focus on the denseness of the values of $\zeta_r(s)$ on a vertical line.
For any $\sigma \in \mathbb{R}$, let $V_{r} (\sigma) = \left\{ \zeta_r \left( \sigma + it \right) \mid t \in \mathbb{R} \right\}$. 

First, we recall the case $r = 1$, which corresponds to the Riemann zeta-function $\zeta_1(s) = \zeta(s)$. 
For $\sigma > 1$, $V_1(\sigma)$ is bounded since $\zeta(\sigma + it)$ converges absolutely. 
Hence $V_1(\sigma)$ is not dense in $\mathbb{C}$. 
For any $1/2 < \sigma \leq 1$, Bohr and Courant \cite{BC1914} proved that $V_1(\sigma)$ is dense in $\mathbb{C}$. 
For any $\sigma < 1/2$,  Garunk\u{s}tis and Steuding \cite{GS2014} showed that $V_1(\sigma)$ is not dense in $\mathbb{C}$ under the Riemann Hypothesis by using the functional equation for $\zeta(s)$. 
As for the case $\sigma = 1/2$, whether the set $V_1(1/2)$ is dense in $\mathbb{C}$ remains open. 
Kowalski and Nikeghbali \cite[Corollary 9]{KN2012} showed that $V_1(1/2)$ is dense in $\mathbb{C}$ under a certain strong assumption with respect to the Keating--Snaith moment conjecture. 
Thus, the denseness problem on the line $\sigma = 1/2$ still appears to be far from being resolved.

Next, we consider the case $r \geq 2$, where the situation is similar to that for $r = 1$. 
For $\sigma > 1$, $V_r(\sigma)$ is bounded since $\zeta_r(\sigma + i t)$ converges absolutely.
Hence $V_r(\sigma)$ is not dense in $\mathbb{C}$ for $\sigma > 1$. 
For any $1/2 < \sigma < 1$, the denseness of $V_r(\sigma)$ can be derived by adapting the method of Nakamura and Pa\' nkowski \cite{NP20161}. 
Although their paper focuses exclusively on the distribution of zeros, their approach can be extended to obtain the denseness result.
They studied $\zeta_r (s)$ by using the harmonic product formula together with the hybrid universality theorem for $\zeta(s)$. 
Their method can also be adapted to the case $\sigma = 1$. 
We refer the reader to \cite{G1979, KK2007, NP2011, NP2012, NP20161, P2010} for the details of the hybrid universality theorem and its application to multivariable zeta-functions. 
For the case $\sigma < 1/2$, we can show $V_r (\sigma)$ is not dense in $\mathbb{C}$ under the Riemann Hypothesis by following an argument similar to Garunk\u{s}tis and Steuding \cite{GS2014}. 
Note that for the case $\sigma \leq 0$, one can obtain the result without the Riemann Hypothesis, which is essentially proved in \cite[Theorem 1.4]{MO2022}.
The proof will be given in Appendix~\ref{append:A}, as Theorem~\ref{thm:NDT}.

By analogy with the Riemann zeta-function, we refer to the line $\sigma=1/2$ as the critical line. 
This was the remaining case.
The first breakthrough in this direction was due to Mishou for $r=2$, 
presented in a 2017 conference talk \cite{M2017}.
By the harmonic product formula, 
we have $\zeta_2( s ) = \left( \zeta(s)^2 - \zeta(2s) \right) / 2$. 
Hence, under the Riemann Hypothesis, substituting
$s=1/2+i\gamma$ gives
\[
\zeta_2 \left( \frac{1}{2} + i\gamma \right)
=-\frac{1}{2}\zeta(1+2i\gamma), 
\] 
where $\gamma$ is the imaginary part of non-trivial zeros of $\zeta(s)$. 
Analyzing the values $\zeta(1 + 2 i \gamma)$, Mishou proved, under the Riemann Hypothesis, that $V_2(1/2)$ is dense in $\mathbb{C}$. 

In this paper, we generalize Mishou's result by combining the method of Nakamura and Pa\' nkowski \cite{NP20161}.

\begin{theorem}\label{thm:main1}
Assume the Riemann Hypothesis.
\begin{enumerate}[{ \rm (i)}]
\item The set
\[
\left\{ \zeta_r \left( \frac{1}{2} + i \gamma \right) \mid \text{\rm $\gamma$ is the imaginary part of non-trivial zeros of $\zeta(s)$} \right\}
\]
is dense in $\mathbb{C}$ for $r \geq 4$. 

\item The set
\[
\left\{ \zeta_3\left( \frac{1}{2} + i \gamma \right) \mid \text{\rm $\gamma$ is the imaginary part of non-trivial zeros of $\zeta(s)$} \right\}
\]
is bounded and hence not dense in $\mathbb{C}$. 
\end{enumerate}
\end{theorem}

The following corollary immediately follows from Theorem~\ref{thm:main1}.

\begin{corollary}
Assume the Riemann Hypothesis. 
Then, for any $r\geq4$, 
the set
\[
\left\{ \zeta_r \left(\frac{1}{2} + i t \right) \mid t \in \mathbb{R} \right\}
\]
is dense in $\mathbb{C}$. 
\end{corollary}

The following theorem plays a crucial role in the proof of Theorem \ref{thm:main1}.

\begin{theorem}\label{thm:main2}
Let $0 < \gamma_1 \leq \gamma_2 \leq \cdots $ denote the positive imaginary parts of the zeros of the Riemann zeta-function, counted with multiplicity.
Then there exists a subsequence $(\gamma_{n_k})_{k = 1}^\infty$ of $\left( \gamma_n \right)_{n = 1}^\infty$ such that, 
for any positive integer $M$, $a \in \mathbb{C}$, $\theta_1, \ldots, \theta_M \in \mathbb{R}$ and $\epsilon>0$, 
we have
\begin{align*}
\liminf_{N \rightarrow \infty} \frac{1}{N} \# \Bigg{\{} 1 \leq k \leq N \mid & \left| \zeta \left( 1 + 2 i \gamma_{n_k} \right) - a \right| < \epsilon \AND \max_{1 \leq m \leq M}  \left\| \frac{\log p_m}{2 \pi}  \gamma_{n_k} - \theta_m \right\| < \epsilon \Bigg{\}} > 0. 
\end{align*}
Here $\# \left( A \right)$ denotes the cardinality of the set $A$, and $p_m$ denotes the $m$-th prime. 
\end{theorem}

The discrete universality theorem with vertical shifts by imaginary parts $\left( \gamma_n \right)_{n = 1}^\infty$ of non-trivial zeros of the Riemann zeta-function was first proved by Garunk\v{s}tis, Laurin\v{c}ikas, and Macaitien\.{e} \cite{GLM2017} under the weak Montgomery conjecture. 
Sourmelidis, Srichan, and Steuding \cite{SSS2022} later removed this assumption by taking a suitable subsequence of $\left( \gamma_n \right)_{n = 1}^\infty$ and obtained a result for $L$-functions in the Selberg class.
Recently, Nakai \cite{N2025} extended the former type of result to Matsumoto zeta-functions under the weak Montgomery conjecture. 
In this paper, we adapt the method in \cite{SSS2022} to prove Theorem \ref{thm:main2}.

\section{Proof of Theorem \ref{thm:main2}}

\subsection{Probabilistic setting and the statement of the limit theorem}

First, we use the following lemma. 

\begin{lemma}[Littlewood \cite{L1924}]
For every large $T>0$, $\zeta(s)$ has a zero $\beta + i \gamma$ satisfying 
\[
\left| T - \gamma \right| < \frac{ 16 }{ \log \log \log T }.
\]
\end{lemma}

The following corollary follows immediately from the above. 

\begin{corollary}\label{cor:LL}
There exists a subsequence $\left( \gamma_{n_k} \right)_{k = 1}^\infty$ of $\left( \gamma_{n} \right)_{n = 1}^\infty$ such that $ \gamma_{n_k} = k + o(1)$ as $k \rightarrow \infty$. 
\end{corollary}

In what follows, we fix a subsequence $\left( \gamma_{n_k} \right)_{k = 1}^\infty$ of $\left( \gamma_{n} \right)_{n = 1}^\infty$ as in Corollary \ref{cor:LL}.

Next, we give a probabilistic setting. 
Let $\mathcal{P}$ denote the set of all prime numbers and let $\mathbb{T} = \{ z \in \mathbb{C} \mid | z | = 1 \}$. 
Define $\Omega = \prod_{p \in \mathcal{P}} \mathbb{T}_{p}$, 
where $\mathbb{T}_p = \mathbb{T}$ for all prime numbers $p$. 
Since $\mathbb{T}_p$ is a compact Hausdorff topological abelian group, 
there exists the probability Haar measure $\bm{m}_p$ on $\left( \mathbb{T}_p, \mathcal{B} \left( \mathbb{T}_p \right) \right)$. 
Using the Kolmogorov extension theorem \cite[Theorem 14.39]{K2020}, 
we can construct the probability Haar measure $\bm{m} = \otimes_{p \in \mathcal{P}} \bm{m}_p$ on $\left( \Omega, \mathcal{B} ( \Omega ) \right)$. 

Let $M$ be a positive integer. 
We define a metric on $\mathbb{C} \times \prod_{m = 1}^M \mathbb{T}_{p_m}$ as follows: 
For $e^{2 \pi i \theta_1 }, e^{2 \pi i \theta_2} \in \mathbb{T}$, 
we define a metric $d_{\mathbb{T}}$ on $\mathbb{T}$ by
\[
d_{\mathbb{T}} \left( e^{2 \pi i \theta_1 }, e^{2 \pi i \theta_2} \right)
= \left\| \theta_1 - \theta_2 \right\|,
\]
where $\left\| x \right\|$ is given by $ \left\| x \right\| = \min_{n \in \mathbb{Z}} \left| x - n \right|$ for $x \in \mathbb{R}$. 
Using this metric, we define a metric on $\mathbb{C} \times \prod_{m = 1}^M \mathbb{T}_{p_m}$ by
\[
d \left( \left( a, (x_{p_m})_{m = 1}^M \right), \left( b, (y_{p_m})_{m = 1}^M \right) \right)
= | a - b | + \sum_{m =1}^M d_{\mathbb{T}} \left( x_{p_m}, y_{p_m} \right)
\]
for $\left( a, (x_{p_m})_{m = 1}^M \right), \left( b, (y_{p_m})_{m = 1}^M \right) \in \mathbb{C} \times \prod_{m = 1}^M \mathbb{T}_{p_m}$. 

Let 
\[
\mathcal{P}_X 
= \left\{ p \in \mathcal{P} \mid \text{$p$ divides $\displaystyle \prod_{n \leq X} n$} \right\}. 
\]
For $\omega \in \Omega$, let $\omega(1) = 1$ and 
\[
\omega(n) = \omega(p_1)^{\nu(n:p_1)} \cdots \omega(p_k)^{\nu(n:p_k)}, 
\]
where $n \geq 2$ and $n = p_1^{\nu(n:p_1)} \cdots p_k^{\nu(n:p_k)}$ is the prime factorization of $n$. 
Then the Menshov--Rademacher theorem \cite[Theorem B.10.5]{K2021} implies that
\[
\zeta(1,\omega^2)
=\sum_{n=1}^{\infty}\frac{\omega(n)^2}{n}
\]
converges almost everywhere.

We define the probability measures $\mu_N$ and $\mu$ on $\left( \mathbb{C} \times \prod_{m = 1}^M \mathbb{T}_{p_m}, \mathcal{B} \left( \mathbb{C} \times \prod_{m = 1}^M \mathbb{T}_{p_m} \right) \right)$ by 
\[
\mu_N(\bm{A}) 
= \frac{1}{N} \# \left\{ 1 \leq k \leq N \mid \left( \zeta \left( 1 + 2 i  \gamma_{n_k} \right), \left( p_m^{ - i  \gamma_{n_k} } \right)_{m = 1}^M \right) \in \bm{A} \right\},
\]
\[
\mu(\bm{A})
= \bm{m} \left( \omega \in \Omega \mid \left( \zeta \left( 1, \omega^2 \right), \left( \omega(p_m) \right)_{m = 1}^M \right) \in \bm{A} \right).
\]

\begin{proposition}\label{prop:LT}
We have the following two statements.
\begin{enumerate}[{\rm (i)}]

\item The probability measures $ \mu_N $ converge weakly to $\mu$ as $N \rightarrow \infty$. 

\item The support of the probability measure $\mu$ coincides with $\mathbb{C} \times \prod_{m = 1}^M \mathbb{T}_{p_m}$. 

\end{enumerate}
\end{proposition}

\subsection{Proof of Proposition \ref{prop:LT} (i)}

\begin{lemma}\label{lem:UD}
Let $\mathcal{P}_F$ be a finite subset of the set of prime numbers and let $h$ be a positive rational number. 
Let $\bm{m}_{\mathcal{P}_F} = \bigotimes_{p \in \mathcal{P}_F} \bm{m}_p$ denote the Haar probability measure on $\prod_{p \in \mathcal{P}_F} \mathbb{T}_p$.
We define the probability measure $\mathbb{H}_{N, \mathcal{P}_F}^{(h)}$ on $\left( \prod_{p \in \mathcal{P}_F} \mathbb{T}_p, \mathcal{B} \left( \prod_{p \in \mathcal{P}_F} \mathbb{T}_p \right) \right)$ by
\[
\mathbb{H}_{N,\mathcal{P}_F}^{(h)} \left( A \right)
= \frac{1}{N} \# \left\{ 1 \leq k \leq N \mid \left( p^{- i h \gamma_{n_k}} \right)_{p \in \mathcal{P}_F} \in A \right\}
\]
for $A \in \mathcal{B}\left( \mathbb{T}^{\mathcal{P}_F} \right)$. 
Then the probability measures $\mathbb{H}_{N, \mathcal{P}_F}^{(h)}$ converge weakly to $\bm{m}_{\mathcal{P}_F}$ as $N \rightarrow \infty$. 
\end{lemma}

\begin{proof}
Let $\phi_N: \mathbb{Z}^{\mathcal{P}_F} \rightarrow \mathbb{C}$ be the Fourier transform of $\mathbb{H}_{N,\mathcal{P}_F}^{(h)}$. 
Then we have
\begin{align*}
\phi_N \left( \left( n_p \right)_{p \in \mathcal{P}_F} \right) 
&= \int_{\mathbb{T}^\mathcal{P}_F} \prod_{p \in \mathcal{P}_F} x_p^{n_p} \, d \mathbb{H}_{N,\mathcal{P}_F}^{(h)} \left( \left( x_p \right)_{p \in \mathcal{P}_F} \right) \\
&= \frac{1}{N} \sum_{1 \leq k \leq N} \prod_{p \in \mathcal{P}_F} p^{- i h \gamma_{n_k} n_p} \\
&= \frac{1}{N} \sum_{1 \leq k \leq N} \exp \left( - i h \gamma_{n_k} \sum_{p \in \mathcal{P}_F} n_p \log p \right).
\end{align*}
By the asymptotic $ \gamma_{n_k} =  k + o(1)$ as $k \rightarrow \infty$, 
we have
\[
\exp \left( - i h \gamma_{n_k} \sum_{p \in \mathcal{P}_F} n_p \log p \right)
= \exp \left( - i h k \sum_{p \in \mathcal{P}_F}  n_p \log p \right) + E(k), 
\]
for $\left( n_p \right)_{p \in \mathcal{P}_F} \neq \bm{0}$, 
where the limit $ E(k) \rightarrow 0 $ holds as $k \rightarrow \infty$. 
Hence for $\left( n_p \right)_{p \in \mathcal{P}_F} \neq \bm{0}$, 
one has
\begin{align*}
&\phi_N \left( \left( n_p \right)_{p \in \mathcal{P}_F} \right) \\
=&\, \frac{1}{N} \sum_{1 \leq k \leq N} \exp \left( - i h k \sum_{p \in \mathcal{P}_F}  n_p \log p \right) 
+ \frac{ \sum_{k = 1}^N E(k)}{N} \\
=&\, \frac{\exp \left( - i h \sum_{p \in \mathcal{P}_F}  n_p \log p \right)}{N} \cdot \frac{1 -  \exp \left( - i h N \sum_{p \in \mathcal{P}_F}  n_p \log p \right) }{1 -  \exp \left( - i h \sum_{p \in \mathcal{P}_F}  n_p \log p \right) } +  \frac{ \sum_{k = 1}^N E(k)}{N} \\
\rightarrow&\, 0
\end{align*}
as $N \rightarrow \infty$, 
since $ h \sum_{p \in \mathcal{P}_F}  n_p \log p \not\in 2 \pi \mathbb{Z}$ by the Gelfond--Schneider theorem. 
For $\left( n_p \right)_{p \in \mathcal{P}_F} = \bm{0}$, 
we have $\phi_N \left( \left( n_p \right)_{p \in \mathcal{P}_F} \right) = 1$. 
This completes the proof. 
\end{proof}

For $X \geq 1$, $s \in \mathbb{C}$, and $\omega \in \Omega$, put
\[
\zeta_X(s) = \sum_{n \leq X} \frac{1}{n^s},
\qquad
\zeta_X(1, \omega^2) = \sum_{n \leq X} \frac{\omega(n)^2}{n}.
\]
We denote the derivative of $\zeta_X(s)$ with respect to $s$ by $\zeta_X'(s)$.

\begin{lemma}\label{lem:MTR}
We have
\[
\lim_{X \rightarrow \infty} \int_{\Omega} \left| \zeta \left( 1, \omega^2 \right) - \zeta_X \left( 1, \omega^2 \right) \right|^2 d \bm{m} (\omega) = 0.
\]
\end{lemma}

\begin{proof}
By Fatou's lemma, one has
\begin{align*}
\mathbb{E}^{\bm{m}} \left[ \left| \zeta \left( 1, \omega^2 \right) - \zeta_X \left( 1, \omega^2 \right) \right|^2 \right]
&=\mathbb{E}^{\bm{m}} \left[ \left|  \sum_{n > X} \frac{\omega(n)^2}{n}  \right|^2 \right] \\
&\leq \liminf_{L \rightarrow \infty}\, \mathbb{E}^{\bm{m}} \left[ \left|  \sum_{X < n \leq L} \frac{\omega(n)^2}{n}  \right|^2 \right] \\
& = \sum_{n > X} \frac{1}{n^2} 
\rightarrow 0
\end{align*}
as $X \rightarrow \infty$. 
This completes the proof. 
\end{proof}

A standard approach is to apply Gallagher's lemma to connect discrete and continuous mean squares of a continuous function. The lemma is stated as follows:

\begin{lemma}[{Gallagher \cite[Lemma 1.4]{M1971}}]\label{lem:GL}
Let $T_0$ and $T \geq \delta >0$ be real numbers, and $\mathcal{T}$ be a finite set in the interval $[T_0 + \delta/2, T_0 + T - \delta/2]$. 
Define
\[
N_{\delta} (x) 
= \sum_{\substack{ t \in \mathcal{T};\\ \left| t - x \right| < \delta }} 1,
\]
and let $S$ be a complex-valued continuous function on $[T_0, T + T_0]$ having a continuous derivative on $(T_0, T + T_0)$. 
Then
\begin{align*}
\sum_{t \in \mathcal{T}} N_\delta(t)^{-1} \left| S (t) \right|^2
\leq \frac{1}{\delta} \int_{T_0}^{T_0 + T} \left| S(x) \right|^2 dx
+ \left( \int_{T_0}^{T_0 + T} \left| S(x) \right|^2 dx \int_{T_0}^{T_0 + T} \left| S'(x) \right|^2 dx \right)^{1/2}. 
\end{align*}

\end{lemma}

\begin{lemma}\label{lem:MV}
We have
\[
\lim_{X \rightarrow \infty} \limsup_{N \rightarrow \infty} \frac{1}{N} \sum_{1 \leq k \leq N} \left| \zeta \left( 1 + 2 i \gamma_{n_k} \right) - \zeta_X \left( 1 + 2 i \gamma_{n_k} \right) \right| = 0. 
\]
\end{lemma}

\begin{proof}
We apply Lemma \ref{lem:GL} with $T_0 = 1$, $\delta = 1/4$, $T = 3N - 1$ and $\mathcal{T} = \{ 2 \gamma_{n_1}, \ldots, 2 \gamma_{n_N} \}$. 
Since 
\[
N_{\delta} (x) 
= \sum_{\substack{ t \in \mathcal{T};\\ \left| t - x \right| < \delta }} 1
\ll 1
\]
for $x \in \mathcal{T}$, 
one has, by the Cauchy--Schwarz inequality,
\begin{align*}
& \left( \frac{1}{N} \sum_{1 \leq k \leq N} \left| \zeta \left( 1 + 2 i \gamma_{n_k} \right) - \zeta_X \left( 1 + 2 i \gamma_{n_k} \right) \right| \right)^2\\
\leq&\, \frac{1}{N} \sum_{1 \leq k \leq N} \left| \zeta \left( 1 + 2 i \gamma_{n_k} \right) - \zeta_X \left( 1 + 2 i \gamma_{n_k} \right) \right|^2 \\
\ll&\, \frac{1}{N} \int_{1}^{3N} \left| \zeta \left( 1 + i t \right) - \zeta_X \left( 1 + i t \right) \right|^2 dt \\
& \quad + \frac{1}{N} \left( \int_{1}^{3N} \left| \zeta \left( 1 + i t \right) - \zeta_X \left( 1 + i t \right) \right|^2 dt \int_{1}^{3N} \left| \zeta' \left( 1 + i t \right) - \zeta_X' \left( 1 + i t \right) \right|^2 dt  \right)^{1/2}. 
\end{align*}
Now, by Carlson's theorem \cite[\S9.51]{T1952}, 
one has
\[
\frac{1}{T} \int_{1}^{T} \left| \zeta \left( 1 + i t \right) - \zeta_X \left( 1 + i t \right) \right|^2 dt 
\sim \sum_{n > X} \frac{1}{n^{2}}
\]
and
\[
\frac{1}{T} \int_{1}^{T} \left| \zeta' \left( 1 + i t \right) - \zeta'_X \left( 1 + i t \right) \right|^2 dt 
\sim \sum_{n > X} \frac{ \left( \log n \right)^2 }{n^{2}}
\]
as $T \rightarrow \infty$. 
From the above asymptotic formulas, we have the conclusion. 
\end{proof}

\begin{proof}[Proof of Proposition \ref{prop:LT} (i)]
We use the method of Kowalski as in \cite[Chapter 3.2]{K2021}. 
For a probability space $(M, \mathcal{F}, \mu)$, the expectation of a random variable $X$ is denoted by $\mathbb{E}^{\mu} [X] = \int_{M} X \, d \mu$. 
Let $F$ be a bounded Lipschitz continuous function on $ \mathbb{C} \times \prod_{m = 1}^M \mathbb{T}_{p_m} $. 
It is enough to show that $\left| \EXP{\mu_N}{F} - \EXP{\mu}{F} \right| \rightarrow 0$ by the Portmanteau theorem (see e.g. \cite[Theorem 13.16]{K2020}). 
By the definition of Lipschitz continuity, 
we can take $C(F)>0$ such that
\[
\left| F \left( a, \left( x_{p_m} \right)_{m = 1}^M \right) - F \left( b, \left( y_{p_m} \right)_{m = 1}^M \right) \right| 
\leq C(F) d \left( \left( a, (x_{p_m})_{m = 1}^M \right), \left( b, (y_{p_m})_{m = 1}^M \right) \right)
\]
for $\left( a, (x_{p_m})_{m = 1}^M \right), \left( b, (y_{p_m})_{m = 1}^M \right) \in \left( \mathbb{C} \times \prod_{m = 1}^M \mathbb{T}_{p_m} \right)$.

Let $\mathbb{P}_N$ denote the probability on the set $\left\{ 1, \ldots, N \right\}$ given by 
\[
\mathbb{P}_N \left( A \right) 
= \frac{1}{N} \# \left \{ 1\leq k \leq N \mid k \in A \right\}
\]
for any $A \subset \left\{ 1, \ldots, N \right\}$. 
We can write
\begin{align*}
&\left| \EXP{\mu_N}{F} - \EXP{\mu}{F} \right| \\
= &\left| \EXP{\mathbb{P}_N}{ F\left( \zeta \left( 1 + 2 i \gamma_{n_k} \right), \left( p_m^{ - i \gamma_{n_k}} \right)_{m = 1}^M \right) } - \EXP{\bm{m}}{F\left( \zeta \left( 1, \omega^2 \right), \left( \omega(p_m) \right)_{m = 1}^M \right)}  \right| \\
\leq&\, \left| \EXP{\mathbb{P}_N}{ F \left( \zeta \left( 1 + 2 i \gamma_{n_k} \right), \left( p_m^{ - i \gamma_{n_k}} \right)_{m = 1}^M \right) } - \EXP{\mathbb{P}_N}{ F\left( \zeta_X \left( 1 + 2 i \gamma_{n_k} \right), \left( p_m^{ - i \gamma_{n_k}} \right)_{m = 1}^M \right) } \right| \\
& \quad + \left| \EXP{\mathbb{P}_N}{ F\left( \zeta_X \left( 1 + 2 i \gamma_{n_k} \right), \left( p_m^{ - i \gamma_{n_k}} \right)_{m = 1}^M \right) } - \EXP{\bm{m}}{ F\left( \zeta_X \left( 1, \omega^2 \right), \left( \omega(p_m) \right)_{m = 1}^M \right) } \right| \\
& \quad \quad + \left| \EXP{\bm{m}}{ F\left( \zeta_X \left( 1, \omega^2 \right), \left( \omega(p_m) \right)_{m = 1}^M \right) } - \EXP{\bm{m}}{ F\left( \zeta \left( 1, \omega^2 \right), \left( \omega(p_m) \right)_{m = 1}^M \right) } \right| \\
=:&\, \Sigma_1(N, X) + \Sigma_2(N, X) + \Sigma_3(N, X),
\end{align*}
for any $X > p_M$. 

As for $\Sigma_1(N, X)$, we have
\begin{align*}
\Sigma_1(N, X)
& \leq C(F) \, \EXP{ \mathbb{P}_N }{ d \left( \left( \zeta \left( 1 + 2 i \gamma_{n_k} \right), \left( p_m^{ - i \gamma_{n_k}} \right)_{m = 1}^M \right),\, \left( \zeta_X \left( 1 + 2 i \gamma_{n_k} \right), \left( p_m^{ - i \gamma_{n_k} } \right)_{m = 1}^M \right) \right) } \\
& = \frac{C(F)}{N} \sum_{1 \leq k \leq N} \left| \zeta \left( 1 + 2 i \gamma_{n_k} \right) - \zeta_X \left( 1 + 2 i \gamma_{n_k} \right) \right|.
\end{align*}
Hence Lemma \ref{lem:MV} implies $ \lim_{X \rightarrow \infty} \limsup_{N \rightarrow \infty} \Sigma_1(N, X) = 0$.

Next, we estimate $\Sigma_2(N, X)$. 
We define the function $\zeta_X \left( 1, \cdot \right): \mathbb{T}^{\mathcal{P}_X } \rightarrow \mathbb{C}$ by
\begin{align*}
\zeta_X \left( 1, \bm{x} \right)
= \sum_{n \leq X} \frac{1}{n} \prod_{p | n} {x_p}^{\nu(n:p)}
\end{align*}
for $\bm{x} \in \mathbb{T}^{\mathcal{P}_X }$.
Then we have
\begin{align*}
&\EXP{\mathbb{P}_N}{ F\left( \zeta_X \left( 1 + 2 i \gamma_{n_k} \right), \left( p_m^{ - i \gamma_{n_k} } \right)_{m = 1}^M \right)} \\
=&\, \EXP{\mathbb{P}_N}{ F\left( \zeta_X \left( 1, \left( \left( p^{ - i \gamma_{n_k}} \right)^2 \right)_{p \in \mathcal{P}_X } \right), \left( p_m^{ - i \gamma_{n_k}} \right)_{m = 1}^M \right)} \\
=&\, \EXP{ \mathbb{H}_{N,\mathcal{P}_X}^{(1)} }{ F \left(  \zeta_X \left( 1, \left( \left( x_p \right)^2 \right)_{p \in \mathcal{P}_X } \right), \left( x_{p_m} \right)_{m = 1}^M  \right) } \\
\rightarrow &\, \EXP{ \bm{m}_{\mathcal{P}_X} }{ F \left(  \zeta_X \left( 1, \left( \left( x_p \right)^2 \right)_{p \in \mathcal{P}_X } \right), \left( x_{p_m} \right)_{m = 1}^M  \right) } \\
=&\, \EXP{\bm{m}}{ F\left( \zeta_X \left( 1, \omega^2 \right), \left( \omega(p_m) \right)_{m = 1}^M \right) }
\end{align*}
as $N \rightarrow \infty$ by the property of weak convergence (see, e.g., \cite[Section 2, The Mapping Theorem]{B1999}).

Finally, we estimate $\Sigma_3(N, X)$. 
One has
\begin{align*}
\Sigma_3(N, X)
&\leq C(F) \, \EXP{ \bm{m} }{ d \left( \left( \zeta_X   \left( 1, \omega^2 \right), \left( \omega(p_m) \right)_{m = 1}^M \right),\, \left( \zeta \left( 1, \omega^2 \right), \left( \omega(p_m) \right)_{m = 1}^M \right)  \right) } \\
&= C(F) \, \int_{\Omega} \left| \zeta_X \left( 1, \omega^2 \right) - \zeta \left( 1, \omega^2 \right) \right| d \bm{m} (\omega) \\
&\leq C(F) \, \left( \int_{\Omega} \left| \zeta_X \left( 1, \omega^2 \right) - \zeta \left( 1, \omega^2 \right) \right|^2 d \bm{m} (\omega) \right)^{1/2}
\end{align*}
by the Cauchy--Schwarz inequality. 
Lemma \ref{lem:MTR} implies $\Sigma_3(N, X) \rightarrow 0$ as $X \rightarrow \infty$.
This completes the proof. 
\end{proof}

\subsection{Proof of Proposition \ref{prop:LT} (ii)}

\begin{lemma}
The random variable 
\[
\log \zeta(1, \omega^2) 
= - \sum_{p} \log \left( 1 - \frac{\omega(p)^2}{p} \right)
\]
converges almost everywhere. 
Moreover, the equality
\[
\zeta(1, \omega^2)
= \exp \left( \log \zeta(1, \omega^2) \right)
\]
holds almost everywhere.
\end{lemma}

\begin{proof}
This can be proved by an argument similar to \cite[Lemma 4.2]{S2007}. 
\end{proof}

\begin{lemma}\label{lem:DL}
The set of all convergent elements
\[
\left( \log \zeta \left(1, c^2 \right), \left( c(p_m) \right)_{m=1}^M \right), \quad c = \left( c(p) \right)_p \in \prod_p \mathbb{T}_p
\]
is dense in $\mathbb{C} \times \prod_{m = 1}^M \mathbb{T}_{p_m}$. 
\end{lemma}

\begin{proof}
Let $a \in \mathbb{C}$, $z(p_1), \ldots, z(p_M) \in \mathbb{T}$, and let $\epsilon > 0$.
Since the squaring map on $\mathbb{T}$ is surjective,
the one-dimensional version of the rearrangement argument as in \cite{V1972} shows that there exists a convergent series
\[
\sum_{m = M + 1}^\infty \log \left( 1 - \frac{c(p_m)^2}{p_m} \right)^{-1}, \quad
c(p_m) \in \mathbb{T}
\]
such that
\[
\left| \sum_{m = M + 1}^\infty \log \left( 1 - \frac{ c(p_m)^2 }{p_m} \right)^{-1} + \sum_{m = 1}^M \log \left( 1 - \frac{ z(p_m)^2 }{p_m} \right)^{-1} - a \right| < \epsilon.
\]
For $1 \leq m \leq M$, we put $c(p_m) = z(p_m)$, and let $c = \left( c(p_m) \right)_{m = 1}^\infty$. 
Then one has
\begin{align*}
&d \left( \left( \log \zeta \left(1, c^2 \right), \left( c(p_m) \right)_{m=1}^M \right),\, \left( a, \left( z (p_m) \right)_{m = 1}^M \right) \right) \\
=&\, \left| \sum_{m = M + 1}^\infty \log \left( 1 - \frac{ c(p_m)^2 }{p_m} \right)^{-1} + \sum_{m = 1}^M \log \left( 1 - \frac{ z(p_m)^2 }{p_m} \right)^{-1} - a \right| 
< \epsilon. 
\end{align*}
This completes the proof. 
\end{proof}

\begin{lemma}[see {\cite[Proposition 5.2]{E2026}}] \label{lem:E2026}
Let $\mathcal{S}$ be a separable metric space and let $\mathcal{B}_j$ be a separable Fr\'echet space for $j = 1, \ldots, r$.
For any $j = 1, \ldots, r$, let $(F_{n,j})_{n = 1}^\infty$ be $\mathcal{B}_j$-valued continuous functions on $\mathcal{S}$. 
Suppose that $\mathcal{M}$ is a finite subset of the set of positive integers and that $X = (X_n)_{n=1}^\infty$ is a sequence of independent $\mathcal{S}$-valued random variables on a probability space $(\Omega, \mathcal{A}, \mathbb{P})$ such that the series $\sum_{n = 1}^\infty F_{n, j} (X_n)$ converges almost surely for $j = 1, \ldots, r$. 
We define the $ \prod_{j = 1}^r \mathcal{B}_j \times \prod_{n \in \mathcal{M}} \mathcal{S}_n $-valued random variable $\Phi$ by
$\Phi (X) = \left( \left( \sum_{n = 1}^\infty F_{n,j}(X_n) \right)_{j = 1}^r,\, (X_n)_{n \in \mathcal{M}} \right)$, where $\mathcal{S}_n = \mathcal{S}$. 
Then the support of the distribution of $\Phi$ coincides with the closure of the set
\[
\left\{ \left( \left( \sum_{n = 1}^\infty F_{n, j}(x_n) \right)_{j = 1}^r ,\, (x_n)_{n \in \mathcal{M}} \right) ~;~ x_n \in \supp( X_n ) \quad {\rm and} \quad \textit{$\displaystyle \sum_{n = 1}^\infty F_{n, j}(x_n)$ {\rm converges}} \right\},
\]
where $\supp(Y)$ stands for the support of the distribution of a random variable $Y$. 
\end{lemma}

\begin{proof}[Proof of Proposition \ref{prop:LT} (ii)]
The assertion follows from Lemma \ref{lem:DL}, Lemma \ref{lem:E2026} and \cite[Lemma B.2.1]{K2021}.

\end{proof}

\subsection{Proof of Theorem \ref{thm:main2}}

\begin{proof}[Proof of Theorem \ref{thm:main2}]

Let $a \in \mathbb{C}$, $\theta_1, \ldots, \theta_M \in \mathbb{R}$ and $\epsilon>0$. 
Put
\[
U = \left\{ u \in \mathbb{C} \mid \left| u - a \right| < \epsilon \right\},
\]
\[
V_{m} 
 =  \left\{ v \in \mathbb{T} \mid d_{\mathbb{T}} \left( v,\, \exp \left( - 2 \pi i \theta_{m} \right)  \right) <  \epsilon \right\} 
 =  \left\{ \exp \left( - 2 \pi i \theta \right) \in \mathbb{T} \mid \left\| \theta - \theta_m \right\| <  \epsilon \right\} 
\]
for $1 \leq m \leq M$. 
Note that $V_{m}$ is an open neighborhood of $\exp \left( - 2 \pi i \theta_{m} \right)$ and that the condition $p_m^{- i \gamma} \in V_{m}$ is equivalent to $\left\| \gamma \left( \log p_m / 2 \pi \right) - \theta_m \right\| < \epsilon$. 
Put $\bm{A} = U \times \left( V_1 \times \cdots \times V_M \right)$. 
Then, by the Portmanteau theorem (see, e.g., \cite[Theorem 13.16]{K2020}) and Proposition \ref{prop:LT}, 
we have
\begin{align*}
&\liminf_{N \rightarrow \infty} \frac{1}{N} \# \Bigg{\{} 1 \leq k \leq N \mid \left| \zeta \left( 1 + 2 i \gamma_{n_k} \right) - a \right| < \epsilon \AND
 \max_{1 \leq m \leq M}  \left\| \frac{\log p_m}{2 \pi}  \gamma_{n_k} - \theta_m \right\| < \epsilon \Bigg{\}} 
 \\ 
= &\, \liminf_{N \rightarrow \infty} \mu_N \left( \bm{A} \right) 
\geq \mu \left( \bm{A} \right) > 0.
\end{align*}
This completes the proof. 
\end{proof}

\section{Proof of Theorem \ref{thm:main1}}

We first recall the Hoffman relation \cite{H1992}.
It follows from the harmonic product relations for multiple zeta values. 
A family $\Pi = \{ P_1, P_2, \ldots, P_\ell \}$ is called a partition of $\left\{ 1, 2, \ldots, r \right\}$ if the sets  $P_1, P_2, \ldots, P_\ell$ are mutually disjoint and $\{ 1, 2, \ldots, r \} = \bigcup_{j=1}^\ell P_j$. 
For any partition $\Pi = \{ P_1, P_2, \ldots, P_\ell \}$ of  $\left\{ 1, 2, \ldots, r \right\}$, we put
\[
c \left( \Pi \right) = (-1)^{r - \ell} \prod_{j = 1}^\ell \left( \left| P_j \right| - 1 \right)! \AND 
\zeta \left( s_1, \ldots, s_r;\, \Pi \right) = \prod_{j = 1}^\ell \zeta \left( \sum_{k \in P_j} s_k \right). 
\]
Using the argument of Hoffman in \cite[Theorem 2.2]{H1992},
one has
\[
\sum_{ \sigma \in S_r} \zeta_r \left( s_{\sigma(1)}, \ldots, s_{\sigma(r)} \right)
= \sum_{\text{partition $\Pi$ of $\{1, 2, \ldots, r\}$}} c \left( \Pi \right) \zeta \left( s_1, \ldots, s_r;\, \Pi \right)
\]
for $s_1, \ldots, s_r \in \mathbb{C}$ with $ \min_{j = 1, \ldots, r} \Re(s_j) > 1$, 
where $S_r$ is the symmetric group of degree $r$. 
Letting $s = s_1 = s_2 = \cdots = s_r$, we have
\begin{equation}\label{eqn:HF}
\zeta_r (s) = \frac{1}{r!} \sum_{\text{partition $\Pi$ of $\{1, 2, \ldots, r\}$}} c \left( \Pi \right) \zeta \left( s ;\, \Pi \right),
\end{equation}
which gives the meromorphic continuation of $\zeta_r(s)$ to the whole complex plane. 
Here we write 
\[
\zeta \left( s ;\, \Pi \right) = \zeta ( \underbrace{s, \ldots, s}_{r} ;\, \Pi ).
\]
Note that Matsumoto, Matsusaka, and Tanackov \cite[Theorem~4.2]{MMT2022}
proved that $\zeta_r(s)$ has poles only at
$s=1/k$, $1\leq k\leq r$, and gave more precise information
about these poles.

\begin{proof}[Proof of Theorem \ref{thm:main1}.]

Assume the Riemann Hypothesis. 
Let $r \geq 4$.
Let $\gamma$ denote the imaginary part of non-trivial zeros of the Riemann zeta-function. 
Let $\mathcal{D}$ denote the set of all partitions. 
We define $\mathcal{A}, \mathcal{B} \subset \mathcal{D}$ by
\begin{gather*}
\mathcal{A} 
= \left\{ \Pi \in \mathcal{D} \mid \text{there exists $P \in \Pi$ such that $| P | = 1$} \right\}, \\
\mathcal{B} = \mathcal{D} \setminus \mathcal{A} 
= \left\{ \Pi \in \mathcal{D} \mid \text{for any $P \in \Pi$, $ | P | \geq 2$} \right\},
\end{gather*}
where $| P |$ denotes the cardinality of $P$. 
For $\Pi \in \mathcal{B}$, we define the polynomial $P_s^{\Pi}(X)$ by 
\[
P_s^{\Pi} (X)
= \left( \prod_{\substack{P \in \Pi; \\ \ |P| \geq 3}} \zeta \left( | P | s \right) \right) X^{K_\Pi},
\]
where $K_\Pi = \# \left\{ P \in \Pi \mid \left| P \right| = 2 \right\}$. 
Then one has $\zeta(s; \Pi) = P_s^{\Pi} ( \zeta( 2s) )$ for $\Pi \in \mathcal{B}$. 
Letting
\[
P_s (X) = \frac{1}{r!} \sum_{\Pi \in \mathcal{B}}  c \left( \Pi \right) P_s^{\Pi} (X) \AND
Q(s) = \frac{1}{r!} \sum_{\Pi \in \mathcal{A}}  c \left( \Pi \right) \zeta (s; \Pi),
\]
we find
\[
\zeta_r(s) = P_{s} \left( \zeta(2s) \right) + Q (s).
\]
Evaluating $s = 1/2 + i \gamma$, 
one has
\[
\zeta_r \left( \frac{1}{2} + i \gamma \right) 
= P_{ 1/2 + i \gamma } \left( \zeta \left( 1 + 2 i \gamma \right) \right),
\]
since every term of $Q(1/2+i \gamma)$ contains the factor $\zeta(1/2+i \gamma)$. 
Let $\epsilon > 0$ and $z \in \mathbb{C}$ be fixed. 
$P_{1/2}$ has positive degree for $r \geq 4$ since the leading term of $P_{1/2} (X)$ equals
\[
\begin{dcases}
\frac{(-1)^{m}}{2^m m!} X^m & \text{if $r = 2m$ for some positive integer $m$}, \\
\frac{(-1)^{m + 1}}{3 \cdot 2^{m -1} (m - 1)!} \zeta(3/2) X^{m-1}& \text{if $r = 2m+1$ for some positive integer $m$}.
\end{dcases}
\]
Hence there exists a solution $a$ of the equation $P_{1/2} (X) = z$, whose existence is guaranteed by the fundamental theorem of algebra. 

By the telescoping argument, 
we have
\[
\sup_{| w - a | \leq 1} \left| P_{1/2 + i \gamma} (w) - P_{1/2} (w) \right|
\leq C(r, a) \sup_{3 \leq \ell \leq r} \left| \zeta \left( \ell \left( \frac{1}{2} + i \gamma \right) \right) - \zeta \left( \frac{\ell}{2} \right) \right|
\]
for some positive constant $C(r, a)$.
By the absolute convergence of the Euler product, 
there exists a positive integer $M_0 = M_0(\epsilon)$ such that for any $M \geq M_0(\epsilon)$
\[
\sup_{x \geq 3/2,\, y \in \mathbb{R}} \left| \zeta\left( x + iy \right) - \prod_{m = 1}^M \left( 1 - \frac{1}{p_m^{x + iy}} \right)^{-1} \right| < \epsilon. 
\]
Using the inequality $\left| \prod_{j = 1}^n a_j - \prod_{j = 1}^n b_j \right|
\leq R^{n-1} \sum_{j=1}^n \left| a_j - b_j \right|$
for $a_1, \ldots, a_n, b_1, \ldots, b_n \in \mathbb{C}$ satisfying $| a_j | \leq R$, $| b_j | \leq R$, $j = 1, \ldots, n$, and the inequalities
\[
\left| e^{ 2 \pi i x} - 1 \right| 
= \left| e^{2 \pi i \| x \| } - 1 \right| 
=\left| 2 \pi i \| x \| \int_{0}^1 e^{2 \pi i \| x \| \theta} d \theta \right| 
\leq 2 \pi \| x \|
\]
and
\[
\| J x\| 
= \min_{n \in \mathbb{Z}} \left| J x - n \right|
= | J | \min_{n \in \mathbb{Z}} \left| x - \frac{n}{J} \right|
\leq | J | \min_{m \in \mathbb{Z}} \left| x - m \right|
= | J | \| x\| 
\]
for $x \in \mathbb{R}$ and $J \in \mathbb{Z}\setminus \{ 0 \}$, 
we have
\[
\max_{3 \leq \ell \leq r} \left| \prod_{m = 1}^M \left( 1 - \frac{ 1 }{p_m^{\ell(1/2 + i \gamma)}} \right)^{-1} - \prod_{m = 1}^M \left( 1 - \frac{ 1 }{p_m^{\ell /2}} \right)^{-1} \right| 
\leq C(M) \max_{1 \leq m \leq M} \left\| \frac{ \log p_m }{2\pi} \gamma \right\| 
\]
for $M \geq M_0$ with some positive number $C(M)$. 
Hence the triangle inequality yields
\[
\max_{3 \leq \ell \leq r} \left| \zeta \left( \ell \left( \frac{1}{2} + i \gamma \right) \right) - \zeta \left( \frac{\ell}{2} \right) \right|
< 3 \epsilon
\]
for $M \geq M_0$ and $ \max_{1 \leq m \leq M} \left\| ( \log p_m / 2\pi) \gamma \right\|  < \epsilon/C(M)$.

On the other hand, using the continuity of the polynomial function $P_{1/2}$, there exists $0 < \delta_0 = \delta_0(r, a, \epsilon) \leq 1$ such that
\[
\sup_{| w - a | < \delta_0 } \left| P_{1/2} ( w ) - P_{1/2} ( a ) \right| < \epsilon.
\]
Therefore putting $\delta_1 = \min\{ \epsilon/C(M), \delta_0 \}$ and applying Theorem \ref{thm:main2}, 
we see that for any $M \geq M_0$, there exists $\gamma = \gamma(\delta_1, M, \epsilon, r)$ such that 
\[
\max_{1 \leq m \leq M} \left\| \frac{ \log p_m }{2 \pi} \gamma \right\| < \delta_1
\AND \left| \zeta \left( 1 + 2 i \gamma \right) - a \right| < \delta_1.
\]
Hence we deduce
\begin{align*}
\left| \zeta_r \left( \frac{1}{2} + i \gamma \right) - z \right|
&\leq \left| P_{ 1/2 + i \gamma } \left( \zeta \left( 1 + 2 i \gamma \right) \right)  - P_{1/2} \left( \zeta \left( 1 + 2 i \gamma \right) \right) \right| 
+ \left| P_{1/2} \left( \zeta \left( 1 + 2 i \gamma \right) \right) - P_{1/2} ( a ) \right| \\
& \leq \left( 3 C(r, a) + 1 \right) \epsilon
\end{align*}
Since $\epsilon > 0$ is arbitrary, we have the asserted denseness result $r \geq 4$.

Next, we consider the case $r = 3$. 
By the formula \eqref{eqn:HF}, 
we have
\[
\zeta_3(s)
= \frac{1}{6} \zeta(s)^{3} - \frac{1}{2} \zeta(2 \, s) \zeta(s) + \frac{1}{3} \zeta(3 \, s).
\]
Substituting $s = 1/2 + i \gamma$, we have $\zeta_3 \left( 1/2 + i \gamma \right) = \zeta \left( 3/2 + 3 i \gamma \right)/3$. 
Therefore one has
\[
\left| \zeta_3 \left( \frac{1}{2} + i \gamma \right)\right|
\leq \frac{1}{3} \zeta \left( \frac{3}{2} \right).
\]
This completes the proof. 
\end{proof}

\appendix

\section{Value-distribution for the region $\sigma < 1/2$}\label{append:A}

\begin{theorem}\label{thm:NDT}
Let $r \geq 2$ be an integer. 
Then we have the following. 
\begin{enumerate}[{\rm (i)}]

\item Let $0 < \sigma < 1/2$ and assume the Riemann Hypothesis. 
Then $\{\zeta_r ( \sigma + it ) \mid t \in \mathbb{R}\}$ is not dense in $\mathbb{C}$. 

\item For $\sigma \leq 0$, $\{\zeta_r ( \sigma + it ) \mid t \in \mathbb{R}\}$ is not dense in $\mathbb{C}$. 

\end{enumerate}

\end{theorem}

To prove the theorem, we quote the following lemma \cite[Lemma 4]{GS2014}.
\begin{lemma}\label{GS-Lem4}
We have the following. 
\begin{enumerate}[{\rm (i)}]

\item Let $0 < \sigma < 1/2$ and assume the Riemann Hypothesis. 
Then for any $\epsilon > 0$ there exists a $c = c(\epsilon, \sigma) > 0$ such that
\[
\left| \zeta \left( \sigma + it \right) \right| 
\geq c | t |^{1/2 - \sigma - \epsilon}
\]
for $| t | \geq 2$. 

\item Let $\sigma \leq 0$.
Then there exists a constant $c=c(\sigma)>0$ such that
\[
 \left|\zeta(\sigma+it)\right|
 >
 \frac{c|t|^{1/2-\sigma}}{\log |t|}
\]
for $|t|\geq 2$. 
\end{enumerate}
\end{lemma}

\begin{proof}[Proof of Theorem \ref{thm:NDT}]
Let $\epsilon$ be a small positive number.
First, we consider the case $0 < \sigma < 1/2$ and assume the Riemann Hypothesis. 
Since the Riemann Hypothesis implies the Lindel\"of Hypothesis, the functional equation gives
\begin{align}\label{est:LH}
\left| \zeta \left( k ( \sigma + it ) \right) \right|
&\ll
\begin{dcases}
\left| t \right|^{1/2 - k \sigma + \epsilon/r} & \text{if $0 \leq k \sigma \leq 1/2$}, \\
\left| t \right|^{\epsilon/r} & \text{if $1/2 < k \sigma$}
\end{dcases} \nonumber \\
& \leq \left| t \right|^{1/2 - \sigma + \epsilon/r}
\end{align}
for $1 \leq k \leq r$ and $| t | \geq 2$.
Let $\mathcal{D}$ denote the set of all partitions of $\left\{ 1, \ldots, r \right\}$ and put $\Pi_0 = \left\{ \{ 1 \}, \ldots, \{r\} \right\}$. 
Then one has
\[
\zeta \left( \sigma + i t; \Pi \right)
\ll | t |^{ \ell (1/2 - \sigma) + \epsilon }
\leq | t |^{ ( r - 1 ) (1/2 - \sigma) + \epsilon }
\]
for any $\Pi = \{ P_1, \ldots, P_\ell \} \in \mathcal{D} \setminus \{ \Pi_0 \}$ and $| t | \geq 2$ by the estimate \eqref{est:LH}, and 
\[
\zeta \left( \sigma + i t; \Pi_0 \right) 
= \zeta ( \sigma + i t )^r
\gg | t |^{ r(1/2 -  \sigma ) - \epsilon }
\]
for $| t | \geq 2$ by the estimate (i) in Lemma \ref{GS-Lem4}. 
Hence we obtain
\begin{align*}
\left| r! \, \zeta_r ( \sigma + i t ) \right|
&\geq \left| \zeta ( \sigma + i t ) \right|^r - \sum_{\Pi \in \mathcal{D} \setminus \{ \Pi_0 \} } \left| c (\Pi) \right| \left| \zeta \left( \sigma + i t; \Pi \right) \right| \\
&\gg | t |^{r (1/2 - \sigma) - \epsilon } + O \left( | t |^{ ( r - 1 ) (1/2 - \sigma) + \epsilon } \right) \\
&= | t |^{r (1/2 - \sigma) - \epsilon } \left( 1 + O\left( | t |^{ - ( 1/2 - \sigma ) + 2 \epsilon } \right) \right)
\end{align*}
for $| t | \geq 2$. 
Therefore there exists $T_0 > 0$ and $C_0 > 0$ such that the estimate $\left| \, \zeta_r ( \sigma + i t ) \right| > C_0$ holds for $| t | > T_0$. 
If $\zeta_r (s)$ does not have any pole at $s = \sigma$, the curve $\zeta_r( \sigma + it )$, $t \in [-T_0, T_0]$ is not dense in the disc $|z| \leq C_0$ since it has finite length. 
If $\zeta_r (s)$ has a pole at $s = \sigma$, we take $\delta > 0$ such that $ \left| \zeta_r( \sigma + it ) \right| > C_0$ for $| t | < \delta$. Then the curve $\zeta_r( \sigma + it )$, $t \in [-T_0, - \delta] \cup [\delta, T_0]$ is not dense in the disc $|z| \leq C_0$ by the same argument. 
This completes the case $0 < \sigma < 1/2$. 

Next, we consider the case $\sigma \leq 0$.
By the functional equation, the Stirling formula and the estimate $\zeta(x + iy) \ll \log | y |$ for $x \geq 1$ and $|y| \geq 2$ (see \cite[Theorem 3.5]{T1986}), we have 
\[
\left| \zeta \left( k ( \sigma + it ) \right) \right|
\ll \left| t \right|^{1/2 - k \sigma} \log | t | 
\ll \left| t \right|^{1/2 - k \sigma  + \epsilon /r}
\]
for $1 \leq k \leq r$ and $| t | \geq 2$. 
Therefore one has
\[
\zeta \left( \sigma + i t; \Pi \right) 
\ll | t |^{ (r-1) / 2 - r  \sigma + \epsilon}
\]
for any $\Pi = \{ P_1, \ldots, P_\ell \} \in \mathcal{D} \setminus \{ \Pi_0 \}$ and $| t | \geq 2$. 
By Lemma \ref{GS-Lem4} (ii), one has
\[
\zeta \left( \sigma + i t; \Pi_0 \right) 
= \zeta ( \sigma + i t )^r
\gg | t |^{ r(1/2 -  \sigma ) - \epsilon }.
\]
Thus, we obtain
\begin{align*}
\left| r! \, \zeta_r ( \sigma + i t ) \right|
&\geq \left| \zeta ( \sigma + i t ) \right|^r - \sum_{\Pi \in \mathcal{D} \setminus \{ \Pi_0 \} } \left| c (\Pi) \right| \left| \zeta \left( \sigma + i t; \Pi \right) \right| \\
&\gg | t |^{r (1/2 - \sigma) - \epsilon} + O \left( | t |^{ (r-1) / 2 - r  \sigma + \epsilon} \right) \\
&= | t |^{r (1/2 - \sigma ) - \epsilon} \left( 1 + O\left( | t |^{ - 1/2 + 2 \epsilon } \right) \right)
\end{align*}
for $| t | \geq 2$ by the estimate (ii) in Lemma \ref{GS-Lem4}.
Hence one can deduce the conclusion from an argument similar to the case $0 < \sigma < 1/2$. 
This completes the proof. 
\end{proof}

\subsection*{Acknowledgments}
I am grateful to Dr. Keita Nakai for bringing several relevant papers to my attention and to Professor Hidehiko Mishou for his valuable advice.
This work was supported by JSPS KAKENHI Grant Number JP25K17244.

\subsection*{Declaration on the Use of Generative AI}
The author used OpenAI's GPT-6 Astra to assist with language editing and with the detection and correction of typographical errors in mathematical expressions.
All mathematical content, including the ideas, results, and proofs, is the author's own, and the author takes full responsibility for the content of this paper.

\vspace{1em}

\begin{flushleft}
{\footnotesize
{\sc
Department of Electronic and Information Engineering, \\
National Institute of Technology, Suzuka College,\\
Shiroko-cho, Suzuka, Mie, 510-0294, Japan
}\\
{\it E-mail address}: {\tt endo-k@info.suzuka-ct.ac.jp}
}
\end{flushleft}

\end{document}